\documentclass{article}
\usepackage[latin1]{inputenc}
\usepackage{fullpage}
\usepackage{quaternions}
\usepackage{amsmath,amsthm}
\usepackage{xspace,tabularx,url,alltt,color}
\newcommand{\matlab}{\textsc{matlab}\textregistered\xspace}
\newcommand{\mat}[1]{\ensuremath{\boldsymbol{#1}}}
\newcommand{\clifford}[2]{\ensuremath{C\kern-0.1em\ell_{#1,#2}}}
\newcommand{\clical}{\textsc{CLICAL}\xspace}
\newtheorem{theorem}{Theorem}
\newtheorem{lemma}{Lemma}
\newtheorem{proposition}{Proposition}
\newcommand{\dif}{\textrm{d}\xspace} 
\title{Complex and Hypercomplex Discrete Fourier Transforms
       Based on Matrix Exponential Form of Euler's Formula}
\author{Stephen~J.~Sangwine\thanks{Stephen~J.~Sangwine is with the
        School of Computer Science and Electronic Engineering,
        University of Essex, Wivenhoe Park,
        Colchester, CO4 3SQ, United Kingdom.
        Email:~sjs@essex.ac.uk}
        \and
        Todd~A.~Ell\thanks{Todd~A.~Ell resides at 5620 Oak View Court, Savage, MN 55378-4695, USA.
        Email:~T.Ell@IEEE.org}
}
\begin{document}
\maketitle
\begin{abstract}
We show that the discrete complex,
and numerous hypercomplex,
Fourier transforms defined and used
so far by a number of
researchers can be unified into a single framework
based on a matrix exponential version of Euler's formula $e^{j\theta}=\cos\theta+j\sin\theta$,
and a matrix root of $-1$ isomorphic to the imaginary root $j$.
The transforms thus defined can be computed
using standard matrix
multiplications and additions with no hypercomplex code,
the complex or hypercomplex algebra being represented
by the form of the matrix root of $-1$,
so that the matrix multiplications
are equivalent to multiplications in the appropriate algebra.
We present examples from the complex, quaternion and biquaternion algebras,
and from Clifford algebras \clifford{1}{1} and \clifford{2}{0}.
The significance of this result is both in the theoretical unification,
and also in the scope it affords for insight into the structure of the various transforms,
since the formulation is such a simple generalization of the classic complex case.
It also shows that hypercomplex discrete Fourier transforms may be {computed}
using standard matrix arithmetic packages {without the need for a hypercomplex library},
which is of importance in providing
a reference implementation for {verifying}
implementations based on hypercomplex code.
\end{abstract}
\section{Introduction}
The discrete Fourier transform
is widely known and used in signal and image processing,
and in many other fields where data is analyzed for frequency content \cite{Bracewell:2000}.
The discrete Fourier transform in one dimension is classically formulated as:
\begin{equation}
\label{eqn:classicdft}
\begin{aligned}
F[u] &= S\sum_{m=0}^{M-1}f[m]\exp\left(          -j 2\pi\frac{mu}{M}\right)\\
f[m] &= T\sum_{u=0}^{M-1}F[u]\exp\left(\phantom{-}j 2\pi\frac{mu}{M}\right)
\end{aligned}
\end{equation}
where $j$ is the imaginary root of $-1$,
$f[m]$ is real or complex valued with $M$ samples,
$F[u]$ is complex valued, also with $M$ samples,
and the two scale factors $S$ and $T$ must multiply to $1/M$.
If the transforms are to be unitary then $S$ must equal $T$ also.
In this paper we {discuss the formulation of the transform}
using a matrix exponential form of Euler's formula in which the imaginary square root of
$-1$ is replaced by {an isomorphic} matrix root.
{This formulation works for the complex DFT, but more importantly,
it works for hypercomplex DFTs (reviewed in §\,\ref{hypercomplex}).}
The matrix exponential formulation is equivalent to all the known hypercomplex
generalizations of the DFT known to the authors, based on quaternion,
biquaternion or Clifford algebras,
through a suitable choice of matrix root of $-1$,
isomorphic to a root of $-1$ in the corresponding hypercomplex algebra.
{All associative hypercomplex algebras (and indeed the complex algebra)
are known to be isomorphic to matrix algebras over the reals or the complex numbers.
For example, Ward \cite[\S\,2.8]{Ward:1997} discusses isomorphism
between the quaternions and $4\times4$ real or $2\times2$ complex matrices
so that quaternions can be replaced by matrices, the rules of matrix
multiplication then being equivalent to the rules of quaternion
multiplication \emph{by virtue of the pattern of
the elements of the quaternion within the matrix}.
Also in the quaternion case, Ickes \cite{Ickes:1970} wrote an important paper
showing how multiplication of quaternions
could be accomplished using a matrix-vector or vector-matrix
product that could accommodate reversal of the product ordering
by a partial transposition within the matrix.
This paper, more than any other, led us to the {observations}
presented here.}

{The fact that a hypercomplex DFT may be formulated using a matrix
exponential may not be surprising.
Nevertheless, to our knowledge, those who have worked on hypercomplex
DFTs have not so far noted or exploited the observations made in this
paper, which is surprising, given the ramifications discussed later.}

\section{{Hypercomplex transforms}}
\label{hypercomplex}
{
The first published descriptions of hypercomplex transforms that we are
aware of date from the late 1980s, using quaternions. In all three known
earliest formulations, the transforms were defined for two-dimensional
signals (that is functions of two independent variables).
The two earliest formulations \cite[§\,6.4.2]{Ernst:1987} and
\cite[Eqn.~20]{Delsuc:1988} are almost equivalent
(they differ only in the placing of the exponentials and the signal
and the signs inside the exponentials)\footnote{{In comparing
the various formulations of hypercomplex transforms, we have changed
the symbols used by the original authors in order to make the comparisons
clearer. We have also made trivial changes such as the choice of basis
elements used in the exponentials.}}:
\[
F(\omega_1,\omega_2) = \int_{-\infty}^{\infty}\int_{-\infty}^{\infty}
                       f(t_1, t_2)
                       e^{\i\omega_1 t_1}e^{\j\omega_2 t_2}\dif t_1\dif t_2
\]
In a non-commutative algebra the ordering of exponentials within
an integral is significant, and of course, two exponentials with
different roots of -1 cannot be combined trivially. Therefore there
are other possible transforms that can be defined by positioning
the exponentials differently.
The first transform in which the exponentials were placed either
side of the signal function was that of Ell \cite{Ell:thesis,Ell:1993}:
\[
F(\omega_1,\omega_2) = \int_{-\infty}^{\infty}\int_{-\infty}^{\infty}
                       e^{\i\omega_1 t_1}
                       f(t_1, t_2)
                       e^{\j\omega_2 t_2}\dif t_1\dif t_2
\]
}%
This style of transform was followed by Chernov, B\"{u}low and Sommer
\cite{Chernov:1995,Bulow:1999,BulowSommer:2001} and others since.
In 1998 the present authors described a single-sided hypercomplex
transform for the first time \cite{SangwineEll:1998} exactly as in
\eqref{eqn:classicdft} except that $f$ and $F$ were quaternion-valued
and $j$ was replaced by a general quaternion root of $-1$.
{Expressed in the same form as the transforms above, this would be:
\[
F(\omega_1,\omega_2) = \int_{-\infty}^{\infty}\int_{-\infty}^{\infty}
                       e^{\mu(\omega_1 t_1 + \omega_2 t_2)}
                       f(t_1, t_2) \dif t_1\dif t_2
\]
where $\mu$ is now an arbitrary root of -1, not necessarily a basis
element of the algebra. The realisation that an arbitrary root of -1
could be used meant that it was possible to define a hypercomplex
transform applicable to one dimension:
\[
F(\omega) = \int_{-\infty}^{\infty}e^{\mu\omega t} f(t) \dif t
\]
}
Pei \textit{et al} have studied efficient implementation of quaternion
FFTs and presented a transform based on commutative reduced
biquaternions \cite{PeiDingChang:2001,PeiChangDing:2004}.
Ebling and Scheuermann defined a Clifford Fourier transform
\cite[\S\,5.2]{10.1109/TVCG.2005.54},
but their transform used the pseudoscalar
(one of the basis elements of the algebra)
as the square root of $-1$.

{A}part from the works by the present authors
\cite{SangwineEll:1998,SangwineEll:2000b,10.1109/TIP.2006.884955},
the idea of using a root of $-1$ different to the basis elements of a
hypercomplex algebra was not developed {further until 2006},
with the publication of a
paper setting out the roots of $-1$ in biquaternions
(a quaternion algebra with complex numbers as the components of the quaternions)
\cite{10.1007/s00006-006-0005-8}.
This work prepared the ground for a biquaternion Fourier transform
\cite{10.1109/TSP.2007.910477}
based on the present authors' one-sided quaternion transform \cite{SangwineEll:1998}.
More recently, the idea of finding roots of $-1$ in other algebras
has been advanced in Clifford algebras by Hitzer and Ab{\l}amowicz
\cite{10.1007/s00006-010-0240-x} with the express intent of using them in
Clifford Fourier transforms, perhaps generalising the ideas of
Ebling and Scheuermann \cite{10.1109/TVCG.2005.54}.
Finally, in this very brief summary of prior work we mention
that the idea of applying hypercomplex algebras in signal
processing has been studied by other authors apart from those
referenced above.
For an overview see \cite{AlfsmannGocklerSangwineEll:2007}.

{In what follows we concentrate on DFTs in one dimension
for simplicity, returning to the two dimensional case in §\,\ref{twodimdft}.}
\section{Matrix formulation of the discrete Fourier transform}
\subsection{Matrix form of Euler's formula}
\label{sec:mateuler}
The transform presented in this paper depends on a generalization of
Euler's formula: $\exp i\theta = \cos\theta + i\sin\theta$,
in which the imaginary root of $-1$ is replaced by a matrix root,
that is, a matrix that squares to give a negated identity matrix.
Even among $2\times2$ matrices there is an infinite number of such
roots \cite[p16]{Nahin:2006}.
In the matrix generalization, the exponential must, of course,
be a matrix exponential \cite[\S\,11.3]{Golub:1996}.
{The following Lemma is not claimed to be original but we have not been
able to locate any published source that we could cite here.
Since the result is essential to Theorem~\ref{theorem:matrixdft},
we set it out in full.}
\begin{lemma}
\label{lemma:euler}
Euler's formula may be generalized as follows:
\begin{equation*}
e^{\mat{J}\theta} = \mat{I}\cos\theta + \mat{J}\sin\theta
\end{equation*}
where \mat{I} is an identity matrix, and $\mat{J}^2 = -\mat{I}$.
\end{lemma}
\begin{proof}
The result follows from the series expansions of the matrix exponential
and the trigonometric functions.
From the definition of the matrix exponential \cite[\S\,11.3]{Golub:1996}:
\begin{align*}
e^{\mat{J}\theta} &= \sum_{k=0}^{\infty} \frac{\mat{J}^k\theta^k}{k!} =
\mat{J}^0 + \mat{J}\theta + \frac{\mat{J}^2\theta^2}{2!} +
                            \frac{\mat{J}^3\theta^3}{3!} + \cdots
\intertext{Noting that $\mat{J}^0=\mat{I}$ (see \cite[Index Laws]{CollinsDictMaths}),
and separating the series into even and odd terms:}
                  &= \mat{I} - \frac{\mat{I}\theta^2}{2!}
                             + \frac{\mat{I}\theta^4}{4!}
                             - \frac{\mat{I}\theta^6}{6!} + \cdots\\
&\phantom{=} + \mat{J}\theta - \frac{\mat{J}\theta^3}{3!}
                             + \frac{\mat{J}\theta^5}{5!}
                             - \frac{\mat{J}\theta^7}{7!} + \cdots\\
                  &= \mat{I}\cos\theta + \mat{J}\sin\theta
\end{align*}
\qed
\end{proof}
Note that matrix versions of the trigonometric functions are not
needed to compute the matrix exponential, because $\theta$ is a scalar.
In fact, if the exponential is evaluated numerically using a
matrix exponential algorithm or function, the trigonometric
functions are not even explicitly evaluated \cite[\S\,11.3]{Golub:1996}.
In practice, given that this is a special case of the matrix exponential,
(because $\mat{J}^2=-\mat{I}$),
it is likely to be numerically preferable to evaluate the
trigonometric functions and to sum scaled versions of \mat{I} and \mat{J}.

{Notice that the matrix $e^{\mat{J}\theta}$ has a structure
with the cosine of $\theta$ on the diagonal and the (scaled) sine of $\theta$
where there are non-zero elements of \mat{J}.}

\subsection{Matrix form of DFT}
\label{sec:matdft}
The classic discrete Fourier transform of \eqref{eqn:classicdft} may be
generalized to a matrix form in which the signals are vector-valued with
$N$ components each and the root of $-1$ is replaced by an $N\times N$ matrix root
$\mat{J}$ such that {$\mat{J}^2=-\mat{I}$}.
In this form, subject to choosing the correct representation for the matrix
root of $-1$, we may represent a wide variety of complex and hypercomplex
Fourier transforms.
\begin{theorem}\label{theorem:matrixdft}
The following are a discrete Fourier transform pair\footnote{The colon notation used
here will be familiar to users of \matlab (an explanation may be found
in \cite[\S\,1.1.8]{Golub:1996}). Briefly, $\mat{f}[:,m]$ means the
$m\textsuperscript{th}$ column of the matrix \mat{f}.}:
\begin{align}
\label{eqn:forward}
\mat{F}[:,u] &= S\sum_{m=0}^{M-1}\exp\left(          -\mat{J}\,2\pi\frac{mu}{M}\right)\mat{f}[:,m]\\
\label{eqn:inverse}
\mat{f}[:,m] &= T\sum_{u=0}^{M-1}\exp\left(\phantom{-}\mat{J}\,2\pi\frac{mu}{M}\right)\mat{F}[:,u]
\end{align}
where \mat{J} is a $N\times N$ matrix root of $-1$,
\mat{f} and \mat{F} are $N\times M$ matrices with one sample per column,
and the two scale factors $S$ and $T$ multiply to give $1/M$.
\end{theorem}
\begin{proof}
\newcommand{\m}{\mathcal{M}}
The proof is based on substitution of the forward transform \eqref{eqn:forward}
into the inverse \eqref{eqn:inverse} followed by algebraic reduction to a result
equal to the original signal $\mat{f}$.
We start by substituting \eqref{eqn:forward} into
\eqref{eqn:inverse}, replacing $m$ by $\m$ to keep the two indices distinct{,
and at the same time replacing the two scale factors by their product $1/M$}:
\begin{equation*}
\mat{f}[:,m] = {\frac{1}{M}}\sum_{u=0}^{M-1}\left[e^{\mat{J}2\pi\frac{mu}{M}}
               \sum_{\m=0}^{M-1}e^{-\mat{J}2\pi\frac{\m u}{M}}\mat{f}[:,\m]\right]
\end{equation*}
The exponential of the outer summation can be moved inside the inner,
because it is constant with respect to the summation index $\m$:
\begin{equation*}
\mat{f}[:,m] = \frac{1}{M}\sum_{u=0}^{M-1}\sum_{\m=0}^{M-1}
               e^{\mat{J}\,2\pi\frac{mu}{M}}
               e^{-\mat{J}\,2\pi\frac{\m u}{M}}
               \mat{f}[:,\m]
\end{equation*}
The two exponentials have the same root of $-1$, namely \mat{J},
and therefore they can be combined:
\begin{equation*}
\mat{f}[:,m] = \frac{1}{M}\sum_{u=0}^{M-1}\sum_{\m=0}^{M-1}
               e^{\mat{J}\,2\pi\frac{(m-\m)u}{M}}
               \mat{f}[:,\m]
\end{equation*}
We now isolate out from the inner summation the case where $m=\m$.
In this case the exponential reduces to an identity matrix, and we have:
\begin{align*}
\mat{f}[:,m] &= \frac{1}{M}\sum_{u=0}^{M-1}\mat{f}[:,m]\\
             &+ \frac{1}{M}\sum_{u=0}^{M-1}\left[\left.
                \sum_{\m=0}^{M-1}\right|_{\m\ne m}
                e^{\mat{J}\,2\pi\frac{(m-\m)u}{M}}
                \mat{f}[:,\m]\right]
\end{align*}
The first line on the right sums to $\mat{f}[:,m]$, which is the original
signal, as required.
To complete the proof, 
we have to show that the second line on the right reduces to zero.
Taking the second line alone,
and changing the order of summation, we obtain:
\[
\left.\sum_{\m=0}^{M-1}\right|_{\m\ne m}
\left[\sum_{u=0}^{M-1}e^{\mat{J}\,2\pi\frac{(m-\m)u}{M}}\right]\mat{f}[:,\m]
\]
Using Lemma~\ref{lemma:euler} we now write the matrix exponential
as the sum of a cosine and sine term.
\begin{equation*}
\left.\sum_{\m=0}^{M-1}\right|_{m\ne\m}
\left[
\begin{aligned}
 \mat{I}&\sum_{u=0}^{M-1}\cos\left(\!2\pi\frac{(m-\m)u}{M}\right)\\
+\mat{J}&\sum_{u=0}^{M-1}\sin\left(\!2\pi\frac{(m-\m)u}{M}\right)
\end{aligned}
\right]\mat{f}[:,\m]
\end{equation*}
Since both of the inner summations are sinusoids summed over an
integral number of cyles, they vanish, and this completes the
proof.
\qed
\end{proof}
Notice that the requirement for $\mat{J}^2=-\mat{I}$ is the only
constraint on \mat{J}.
It is not necessary to constrain elements of \mat{J} to be real. 
Note that $\mat{J}^2=\mat{-I}$ implies that $\mat{J}^{-1}=-\mat{J}${,}
hence the inverse transform is obtained by negating or inverting the matrix
root of $-1$ (the two operations are equivalent).

The matrix dimensions must be consistent according to the ordering inside the summation.
As written above, for a complex transform represented in matrix
form,
\mat{f} and \mat{F} must have two rows and $M$ columns.
If the exponential were to be placed on the right,
\mat{f} and \mat{F} would have to be transposed,
with two columns and $M$ rows.

It is important to realize that \eqref{eqn:forward} is totally
different to the classical matrix formulation of the discrete Fourier
transform,
as given for example by Golub and Van Loan \cite[\S\,4.6.4]{Golub:1996}.
The classic DFT given in \eqref{eqn:classicdft} can be formulated as a matrix
equation in which a large $M\times M$  {Vandermonde} matrix containing
$n\textsuperscript{th}$ roots of unity
multiplies the signal $f$ expressed as a vector of real or complex values.
Instead,
in \eqref{eqn:forward} each matrix exponential multiplies
\emph{one column} of \mat{f},
corresponding to \emph{one sample} of the signal represented by \mat{f}
and the dimensions of the matrix exponential are set by the dimensionality
of the algebra (2 for complex, 4 for quaternions \textit{etc.}).
In \eqref{eqn:forward} it is the multiplication of the exponential and
the signal samples, dependent on the algebra involved,
that is expressed in matrix form,
not the structure of the transform itself.

Readers who are already familiar with hypercomplex Fourier transforms should note that
the ordering of the exponential within the summation \eqref{eqn:forward}
is not related to the ordering within the hypercomplex formulation of the transform
(which is significant because of non-commutative multiplication).
The hypercomplex ordering can be accommodated within the framework presented
here by changing the representation of the matrix root of $-1$,
in a non-trivial way, shown for the quaternion case by
Ickes \cite[Equation 10]{Ickes:1970} and called \emph{transmutation}.
We have studied the generalisation of Ickes' transmutation to the case
of Clifford algebras, and it appears {that there is a more general}
operation. {In the cases we have studied this can be}
described as negation of the off-diagonal elements of the
lower-right sub-matrix, excluding the first row and column\footnote{This
gives the same result as transmutation in the quaternion case.}.
We believe a more general result is known in Clifford algebra
but we have not been able to locate a clear statement that we could cite.
We therefore leave this for later work, as a full generalisation to Clifford algebras of
arbitrary dimension requires further work, and is more appropriate to
a more mathematical paper.

\section{Examples in specific algebras}
\label{sec:examples}
In this section we present the information necessary for \eqref{eqn:forward}
and \eqref{eqn:inverse} to be verified numerically.
In each of the cases below, we present an example root of $-1$ and
a matrix representation{\footnote{{The matrix representations of
roots of -1 are not unique -- a transpose of the matrix, for example, is
equally valid. The operations that leave the square of the matrix invariant
probably correspond to fundamental operations in the hypercomplex algebra, for
example negation, conjugation, reversion.}}}.
We include in the Appendix a short \matlab function for computing
the transform in \eqref{eqn:forward}.
The same code will compute the inverse by negating \mat{J}.
This may be used to verify the results in the next section and to
compare the results obtained with the classic complex FFT.
In order to verify the quaternion or biquaternion results,
the reader will need to install the QTFM toolbox \cite{qtfm}{,
or use some other specialised software for computing with quaternions.}

\subsection{Complex algebra}
\label{sec:complex}
The $2\times2$ real matrix
$\left(
\begin{smallmatrix}
0 &           -1\\
1 & \phantom{-}0
\end{smallmatrix}
\right)$
can be easily verified by eye to be a square root of the negated identity matrix
$\left(\begin{smallmatrix}-1 & \phantom{-} 0\\ \phantom{-} 0 &-1\end{smallmatrix}\right)$,
and it is easy to verify numerically
that Euler's formula gives the same numerical results in the classic complex
case and in the matrix case for an arbitrary $\theta$.
This root of $-1$ is based on the well-known isomorphism between a complex number
$a+j b$ and the matrix representation
$\left(\begin{smallmatrix}a & -b\\b & \phantom{-}a\end{smallmatrix}\right)$
\cite[Theorem 1.6]{Ward:1997}\footnote{We have used the transpose of Ward's
representation for consistency with the quaternion and biquaternion representations
in the two following sections.}.
The structure of a matrix exponential $e^{\mat{J}\theta}$ using the above matrix
for \mat{J} is
$\left(
\begin{smallmatrix}
C &           -S\\
S & \phantom{-}C
\end{smallmatrix}
\right)$ where $C=\cos\theta$ and $S=\sin\theta$.

\subsection{Quaternion algebra}
\label{sec:quaternion}
The quaternion roots of $-1$ were discovered by Hamilton
\cite[pp\,203, 209]{Hamiltonpapers:V3:7}, and consist of all
unit pure quaternions, that is quaternions of the form $x\i+y\j+z\k$
subject to the constraint $x^2+y^2+z^2=1$.
A simple example is the quaternion $\boldsymbol\mu=(\i+\j+\k)/\sqrt{3}$,
which can be verified by hand to be a square root of $-1$ using the usual
rules for multiplying the quaternion basis elements ($\i^2=\j^2=\k^2=\i\j\k=-1$).
Using the isomorphism with $4\times4$ matrices given by Ward \cite[\S\,2.8]{Ward:1997},
between the quaternion
$w+x\i+y\j+z\k$ and the matrix:
\begin{align*}
&
 \begin{pmatrix}
 w &           -x &           -y &           -z\\
 x & \phantom{-}w &           -z & \phantom{-}y\\
 y & \phantom{-}z & \phantom{-}w &           -x\\
 z &           -y & \phantom{-}x & \phantom{-}w
 \end{pmatrix}
\intertext{we have the following matrix representation:}
\boldsymbol\mu = \frac{1}{\sqrt{3}}
&
 \begin{pmatrix}
 \newcommand{\s}{{\color{white}+}}
 0 &           -1 &           -1 &           -1\\
 1 & \phantom{-}0 &           -1 & \phantom{-}1\\
 1 & \phantom{-}1 & \phantom{-}0 &           -1\\
 1 &           -1 & \phantom{-}1 & \phantom{-}0
 \end{pmatrix}
\end{align*}
Notice the structure that is apparent in this matrix: the
$2\times2$ blocks on the leading diagonal {at the top left and bottom right}
can be recognised as roots of $-1$ in the complex algebra as shown in \S\,\ref{sec:complex}
\begin{proposition}
\label{prop:quatroot}
Any matrix of the form:
\[
\begin{pmatrix}
 0 &           -x &           -y &           -z\\
 x & \phantom{-}0 &           -z & \phantom{-}y\\
 y & \phantom{-}z & \phantom{-}0 &           -x\\
 z &           -y & \phantom{-}x & \phantom{-}0
\end{pmatrix}
\]
with $x^2+y^2+z^2=1$ is the square root of a negated $4\times4$ identity matrix.
Thus the matrix representations of the quaternion roots of $-1$ are all
roots of the negated $4\times4$ identity matrix.
\end{proposition}
\begin{proof}
The matrix is anti-symmetric, and the inner product of the $i\textsuperscript{th}$
row and $i\textsuperscript{th}$ column is $-x^2-y^2-z^2$, which is $-1$
because of the stated constraint.
Therefore the diagonal elements of the square of the matrix are $-1$.
Note that the rows of the matrix have one or three negative values,
whereas the columns have zero or two.
The product of the $i\textsuperscript{th}$ row with the $j\textsuperscript{th}$
column, $i\ne j$, is the sum of two values of opposite sign and equal magnitude.
Therefore all off-diagonal elements of the square of the matrix are zero.
\qed
\end{proof}
The structure of a matrix exponential $e^{\mat{J}\theta}$ using a matrix as in
Proposition \ref{prop:quatroot} for \mat{J} is: 
\[
\begin{pmatrix}
\phantom{x}C &           -xS &           -yS &           -zS\\
          xS & \phantom{-x}C &           -zS & \phantom{-}yS\\
          yS & \phantom{-}zS & \phantom{-x}C &           -xS\\
          zS &           -yS & \phantom{-}xS & \phantom{-x}C
\end{pmatrix}
\]
where, as before, $C=\cos\theta$ and $S=\sin\theta$.
\subsection{Biquaternion algebra}
The biquaternion algebra \cite[Chapter 3]{Ward:1997}
(quaternions with complex elements)
can be handled exactly as in the previous section, except
that the $4\times4$ matrix representing the root of $-1$ must
be complex (and the signal matrix must have {four}
complex elements {per column}).
The set of square roots of $-1$ in the biquaternion algebra
is given in \cite{10.1007/s00006-006-0005-8}.
A simple example is $\i+\j+\k+\I(\j-\k)$ (where \I denotes the
classical complex root of $-1$, that is the biquaternion has
real part $\i+\j+\k$ and imaginary part $\j-\k$).
Again, this can be verified by hand to be a root of $-1$
and its matrix representation is:
\[
\begin{pmatrix}
0     &           -1     & -1 - \I &           -1 + \I\\
1     & \phantom{-}0     & -1 + \I & \phantom{-}1 + \I\\
1 +\I & \phantom{-}1 -\I &  0     &           -1\\          
1 -\I &           -1 -\I &  1     & \phantom{-}0          
\end{pmatrix}
\]
Again, sub-blocks of the matrix have recognizable structure.
The {upper left and lower right} diagonal $2\times2$ blocks are roots of
$-1$, while the {lower left and upper right} off-diagonal {$2\times2$}
blocks are nilpotent -- that is their square vanishes.

\subsection{Clifford algebras}
Recent work by Hitzer and Ab{\l}amowicz \cite{10.1007/s00006-010-0240-x}
has explored the roots of $-1$ in Clifford algebras \clifford{p}{q} up to those
with $p+q = 4$, which are 16-dimensional algebras\footnote{$p$ and $q$ are
non-negative integers such that $p+q=n$ and $n\ge1$. The dimension of the
algebra (strictly the dimension of the space spanned by the basis
elements of the algebra) is $2^n$ .}.
The derivations of
the roots of -1 for the 16-dimensional algebras are long and difficult.
Therefore, for the moment, we confine the discussion here to lower-order
algebras, noting that, since all Clifford algebras are isomorphic to a
matrix algebra, we can be assured that if roots of -1 exist, they must
have a matrix representation.
Using {the results obtained by Hitzer and Ab{\l}amowicz},
and by finding from first principles
the layout of a real matrix isomorphic to a Clifford
multivector in a given algebra, it has been possible to
verify that the transform formulation presented in this
paper is applicable to at least the lower order Clifford
algebras.
The quaternions and biquaternions are isomorphic to the
Clifford algebras \clifford{0}{2} and \clifford{3}{0}
respectively so this is not surprising. Nevertheless, it is
an important finding, because until now quaternion and
Clifford Fourier transforms were defined in different
ways, using different terminology, and it was difficult
to make comparisons between the two.
Now, with the matrix exponential formulation,
it is possible to handle quaternion and Clifford transforms
(and indeed transforms in different Clifford algebras) within
the same algebraic and/or numerical framework.

We present examples here from two of the 4-dimensional
Clifford algebras, namely \clifford{1}{1} and \clifford{2}{0}.
These results have been verified against the \clical
package \cite{CLICAL-User-manual} to ensure that the
multiplication rules have been followed correctly and
that the roots of $-1$ found by Hitzer and Ab{\l}amowicz
are correct.

Following the notation in \cite{10.1007/s00006-010-0240-x},
we write a multivector in \clifford{1}{1} as
$\alpha + b_1 e_1 + b_2 e_2 + \beta e_{12}$,
where $e_1^2=+1, e_2^2=-1, e_{12}^2=+1$ and $e_1 e_2 = e_{12}$. 
A possible real matrix representation is as follows:
\[
\begin{pmatrix}
\alpha & \phantom{-}b_1    &              -b_2 & \beta\\
b_1    & \phantom{-}\alpha &            -\beta & b_2\\
b_2    &            -\beta & \phantom{-}\alpha & b_1\\          
\beta  &              -b_2 & \phantom{-}b_1    & \alpha          
\end{pmatrix}
\]
In this algebra, the constraints on the coefficients
of a multivector for it to be a root of $-1$ are as
follows: $\alpha=0$ and $b_1^2-b_2^2+\beta^2=-1$
\cite[Table 1]{10.1007/s00006-010-0240-x}\footnote{We have
re-arranged the constraint compared to \cite[Table 1]{10.1007/s00006-010-0240-x}
to make the comparison with the quaternion case easier:
we see that the signs of the squares of the coefficients
match the signs of the squared basis elements.}.
Choosing $b_1=\beta=1$ gives $b_2=\sqrt{3}$ and thus
$e_1 + \sqrt{3}e_2 + e_{12}$ which can be verified
algebraically or in \clical to be a root of $-1$.
The corresponding matrix is then:
\[
\begin{pmatrix}
0        & \phantom{-}1 &    -\sqrt{3} & 1\\
1        & \phantom{-}0 &           -1 & \sqrt{3}\\
\sqrt{3} &           -1 & \phantom{-}0 & 1\\          
1        &    -\sqrt{3} & \phantom{-}1 & 0          
\end{pmatrix}
\]
Following the same notation in algebra \clifford{2}{0},
in which $e_1^2=e_2^2=+1, e_{12}^2=-1$, a possible
matrix representation is:
\[
\begin{pmatrix}
\alpha & \phantom{-}b_1    &    b_2 & -\beta\\
b_1    & \phantom{-}\alpha &  \beta & -b_2\\
b_2    &            -\beta & \alpha & \phantom{-}b_1\\          
\beta  &              -b_2 &   b_1  & \phantom{-}\alpha          
\end{pmatrix}
\]
The constraints on the coefficients are $\alpha=0$ and
$b_1^2+b_2^2-\beta^2=-1$,
and choosing $b_1=b_2=1$ gives $\beta=\sqrt{3}$ and thus
$e_1 + e_2 + \sqrt{3}e_{12}$ is a root of $-1$.
The corresponding matrix is then:
\[
\begin{pmatrix}
0        & \phantom{-}1  &       1  &     -\sqrt{3}\\
1        & \phantom{-}0  & \sqrt{3} &           -1\\
1        &     -\sqrt{3} &       0  & \phantom{-}1\\          
\sqrt{3} &           -1  &       1  & \phantom{-}0          
\end{pmatrix}
\]
Notice that in both of these algebras the matrix representation
of a root of $-1$ is very similar to that given for the quaternion
case in Proposition~\ref{prop:quatroot}, with zeros on the
leading diagonal, an odd number of negative values in each
row and an even number in each column.
It is therefore simple to see that minor modifications to
Proposition~\ref{prop:quatroot} would cover these algebras
and the matrices presented above.

\section{An example not based on a specific\\algebra}
\label{sec:mystery}
We show here using an arbitrary $2\times2$ matrix root of $-1$,
that it is possible to define a Fourier transf{or}m
{without a specific} algebra.
Let an arbitrary real matrix be given as
$J = \left(\begin{smallmatrix}a & b\\c & d\end{smallmatrix}\right)$,
then by brute force expansion of $J^2=-{I}$ we find
the original four equations reduce to but two independent equations.
Picking $(a,b)$ and solving for the remaining coefficients we find that
any matrix of the form:
\begin{equation*}
\begin{pmatrix}a & \phantom{-} b\\ -(1+a^2)/b &-a\end{pmatrix}
\end{equation*}
with finite $a$ and $b$, and $b\ne0$, is a root of $-1$. Choosing instead $(a,c)$ we get the transpose form:
\begin{equation*}
\begin{pmatrix}a & -(1+a^2)/c \\ c &-a\end{pmatrix}
\end{equation*}
where $c\ne0$.
Choosing the cross-diagonal terms $(b,c)$ yields: 
\begin{equation}
\label{eqn:ellipseroot}
{
\begin{pmatrix}
\pm\kappa &         b\\
       c  & \mp\kappa
\end{pmatrix}
}
\end{equation}
{where $\kappa=\sqrt{ -1 - bc}$ and} $bc\leq-1$. 

In all cases the resulting matrix has eigenvalues of $\lambda  = \pm i$.
(This is a direct consequence of the fact that this matrix squares to $-1$.)
Each form, however, has different eigenvectors.
The standard matrix {representation} for the complex operator {$i$ is}
{$\left(\begin{smallmatrix}0 & -1\\1 &\phantom{-}0\end{smallmatrix}\right)$}
{with}
eigenvectors $v = [1,\pm\,i]$.
In the matrix with $(a,b)$ parameters the eigenvectors are $v = [1,-b/(a\,\pm\,i)]$ 
whereas the cross-diagonal form with $(b,c)$ parameters has eigenvectors
{$v = [1,(\kappa\,\pm\,i)/c]$}.

These forms suggest the interesting question:
which algebra, if any, applies here\footnote{It is possible that
there is no corresponding `algebra' in the usual sense.
Note that there are only two Clifford algebras of dimension 2,
one of which is the algebra of complex numbers.
The other has no multivector roots of -1 \cite[\S\,4]{10.1007/s00006-010-0240-x}
and therefore the roots of $-1$ given above cannot be a root of $-1$ in any
Clifford algebra.}; and how can the Fourier coefficients (the `spectrum') be interpreted? 
{We are not able to answer the first question in this paper.
The `interpretation' of the spectrum is relatively simple.
Consider a spectrum $\mat{F}$ containing only one non-zero column at index $u_0$
with value $\left(\begin{smallmatrix}x\\y\end{smallmatrix}\right)$ and invert this
spectrum using \eqref{eqn:inverse}.
Ignoring the scale factor, the result will be the signal:
\[
\mat{f}[:,m] = \exp\left(\mat{J}\,2\pi\frac{mu_0}{M}\right)\begin{pmatrix}x\\y\end{pmatrix}
\]
The form of the matrix exponential depends on \mat{J}.
In the classic complex case, as given in \S\,\ref{sec:complex},
the matrix exponential, as already seen, takes the form:
\[
\begin{pmatrix}
\cos\theta &          - \sin\theta\\
\sin\theta & \phantom{-}\cos\theta
\end{pmatrix}
\]
where $\theta=2\pi\frac{mu_0}{M}$.
}
This is a rotation matrix and it
maps a real unit vector $\left(\begin{smallmatrix}1\\0\end{smallmatrix}\right)$
to a point on a circle in the complex plane.
It embodies the standard \emph{phasor} concept associated with sinusoidal functions.
Using the same analysis, this time using the matrix in \eqref{eqn:ellipseroot} above,
one obtains for the matrix exponential the `phasor' operator:
\[
\left(
\begin{array}{@{}rr@{}} 
 \cos\theta + \kappa\sin\theta  &                   b\sin\theta\\
                    c\sin\theta & \cos\theta - \kappa\sin\theta
\end{array}
\right)
\]
Instead of mapping a real unit vector
$\left(\begin{smallmatrix}1\\0\end{smallmatrix}\right)$
to a {point on a circle},
this matrix maps to an ellipse.
Thus, we see that a transform based on a matrix such as that
in \eqref{eqn:ellipseroot} has basis functions that are projections of an
elliptical, rather than a circular path in the complex plane, as in the
classical complex Fourier transform.
We refer the reader to a discussion on a similar point for the one-sided quaternion
discrete Fourier transform in our own 2007 paper \cite[\S\,VI]{10.1109/TIP.2006.884955},
in which we showed that the quaternion coefficients of the Fourier spectrum
also represent elliptical paths through the space of the signal samples.

{It is possible that the matrices discussed in this section could be
transformed by similarity transformations into matrices representing elements
of a Clifford algebra\footnote{{We are grateful to Dr Eckhard Hitzer for pointing
this out, in September 2010.}}.
Note that in the quaternion case, any root of -1
lies on the unit sphere in 3-space,
and can therefore be transformed into another root of -1 by a rotation.
It is possible that the same applies in other algebras,
the transformation needed being dependent on the geometry.}
Clearly there are interesting issues to be studied here,
and further work to be done.

\section{Non-existence of transforms in algebras with odd dimension}
\label{sec:nonexist}
In this section we show that there are no real matrix roots of $-1$
with odd dimension.
This is not unexpected, since the existence of such roots would {suggest} 
the existence of a hypercomplex algebra of odd dimension.
The significance of this result is to show that there is no
discrete Fourier transform
as formulated in Theorem \ref{theorem:matrixdft}
for an algebra of dimension $3$, which is
of importance for the processing of signals representing physical 3-space
quantities, or the values of colour image pixels.
We thus conclude that the choice of quaternion Fourier transforms or
a Clifford Fourier transform of dimension $4$ is inevitable in these
cases.
This is not an unexpected conclusion, nevertheless,
in the experience of the authors, some researchers in signal and image
processing hesitate to accept the idea of using four dimensions to
handle three-dimensional samples or pixels. (This is despite the rather
obvious parallel of needing two dimensions -- complex numbers -- to
represent the Fourier coefficients of a real-valued signal or image.)
\begin{theorem}
There are no $N\times N$ matrices $\mat{J}$ with real elements such that
$\mat{J}^2=-\mat{I}$ for odd values of $N$.
\end{theorem}
\begin{proof}
The determinant of a diagonal matrix is the product of its diagonal entries.
Therefore $|-\mat{I}|=-1$ for odd $N$.
Since the product of two determinants is the determinant of the product,
$|\mat{J}^2|=-1$ requires $|\mat{J}|^2=-1$, which cannot
be satisfied if $\mat{J}$ has real elements.
\qed
\end{proof}

\section{Extension to two-sided DFTs}
\label{twodimdft}
There have been various definitions of two sided hypercomplex
Fourier transforms and DFTs.
We consider here only one case
to demonstrate that the approach presented in this
paper is applicable to two-sided as well as one-sided transforms: this is
a matrix exponential Fourier transform based on
Ell's original two-sided two-dimensional
quaternion transform
\cite[Theorem 4.1]{Ell:thesis}, \cite{Ell:1993}, \cite{10.1049/el:19961331}.
A more general formulation is:
\begin{equation}
\mat{F}[u, v] = S\!\sum_{m=0}^{M-1}\sum_{n=0}^{N-1}\!
e^{-\mat{J} 2\pi\frac{mu}{M}}\mat{f}[m,n]e^{-\mat{K} 2\pi\frac{nv}{N}}
\label{eqn:twodforward}
\end{equation}
\begin{equation}
\mat{f}[m, n] = T\!\sum_{u=0}^{M-1}\sum_{v=0}^{N-1}\!
e^{+\mat{J} 2\pi\frac{mu}{M}}\mat{F}[u,v]e^{+\mat{K} 2\pi\frac{nv}{N}}
\label{eqn:twodinverse}
\end{equation}
in which \emph{each element} of the two-dimensional arrays \mat{F}
and \mat{f} is a {square} matrix representing a
complex or hypercomplex number using
a matrix isomorphism for the algebra in use,
for example the representations already given in \S\,\ref{sec:quaternion}
in the case of the quaternion algebra;
the two scale factors multiply to give $1/MN$,
and \mat{J} and \mat{K} are matrix
representations of two \emph{arbitrary}
roots of $-1$ in the chosen algebra.
(In Ell's original formulation, the roots of $-1$
were $\j$ and $\k$, that is two of the \emph{orthogonal} quaternion basis
elements. The following theorem shows that there is no requirement for
the two roots to be orthogonal in order for the transform to invert.)
\begin{theorem}
\label{theorem:twodmatrixdft}
The transforms in \eqref{eqn:twodforward} and \eqref{eqn:twodinverse}
are a two-dimensional discrete Fourier transform pair,
{provided that $\mat{J}^2 = \mat{K}^2 = -\mat{I}$.}
\end{theorem}
\begin{proof}
\newcommand{\m}{\mathcal{M}}
\newcommand{\n}{\mathcal{N}}
The proof follows the same scheme as the proof of Theorem~\ref{theorem:matrixdft},
but we adopt a more concise presentation to fit the available column space.
We start by substituting \eqref{eqn:forward} into
\eqref{eqn:inverse}, replacing $m$ and $n$ by $\m$ and $\n$ respectively
to keep the indices distinct:
\begin{align*}
\mat{f}[m, n] &= {\frac{1}{M N}}\sum_{u=0}^{M-1}\sum_{v=0}^{N-1}
                e^{\mat{J} 2\pi\frac{mu}{M}}\\
                &\times\left[\sum_{\m=0}^{M-1}\sum_{\n=0}^{N-1}
                e^{-\mat{J} 2\pi\frac{\m u}{M}}
                \mat{f}[\m,\n]
                e^{-\mat{K} 2\pi\frac{\n v}{N}}\right]\\
                &\times
                e^{\mat{K} 2\pi\frac{nv}{N}}
\end{align*}
The scale factors can be moved outside both summations,
and replaced with their product $1/M N$;
and the exponentials of the outer summations can be moved inside the inner,
because they are constant with respect to the summation indices $\m$ and $\n$.
At the same time, adjacent exponentials with the same root of $-1$ can be merged.
With these changes{, and omitting the scale factor to save space},
the right-hand side of the equation becomes:
\begin{equation*}
\sum_{u=0}^{M-1}\sum_{v=0}^{N-1}
\sum_{\m=0}^{M-1}\sum_{\n=0}^{N-1}
e^{\mat{J} 2\pi\frac{(m-\m)u}{M}}
\mat{f}[\m,\n]
e^{\mat{K} 2\pi\frac{(n-\n)v}{N}}
\end{equation*}
We now isolate out from the inner pair of summations the case where $\m=m$ and $\n=n$.
In this case the exponentials reduce to identity matrices, and we have:
\begin{equation*}
\frac{1}{MN}\sum_{u=0}^{M-1}\sum_{v=0}^{N-1}\mat{f}[m,n]
\end{equation*}
This sums to $\mat{f}[m,n]$, which is the original two-dimensional signal,
as required.
To complete the proof we have to show that the rest of the summation,
excluding the case $\m=m$ and $\n=n$, reduces to zero.
Dropping the scale factor, and changing the order of summation,
we have the following inner double summation:
\begin{equation*}
\sum_{u=0}^{M-1}\sum_{v=0}^{N-1}
e^{\mat{J} 2\pi\frac{(m-\m)u}{M}}
\mat{f}[\m,\n]
e^{\mat{K} 2\pi\frac{(n-\n)v}{N}}
\end{equation*}
Noting that the first exponential and \mat{f} are independent
of the second summation index $v$, we can move them outside
the second summation (we could do similarly with the exponential
on the right and the first summation):
\begin{equation*}
\sum_{u=0}^{M-1}
e^{\mat{J} 2\pi\frac{(m-\m)u}{M}}
\mat{f}[\m,\n]
\sum_{v=0}^{N-1}
e^{\mat{K} 2\pi\frac{(n-\n)v}{N}}
\end{equation*}
and, as in Theorem \ref{theorem:matrixdft},
the summation on the right is over an integral number of cycles
of cosine and sine, and therefore vanishes.
\qed
\end{proof}
Notice that it was not necessary to assume that \mat{J} and \mat{K}
were orthogonal: it is sufficient that each be a root of $-1$.
This has been verified numerically using the two-dimensional code
given in the Appendix.

\section{Discussion}
We have shown that any discrete Fourier transform in an algebra
that has a matrix representation, can be formulated in the way
shown here.
This includes the complex, quaternion, biquaternion, and Clifford
algebras (although we have demonstrated only certain cases of
Clifford algebras, we believe the result holds in general).
{This observation provides a theoretical unification of
diverse hypercomplex DFTs.}

Several immediate possibilities for further work, as well as
ramifications, now suggest themselves.
Firstly,
the study of roots of $-1$ is accessible from the matrix
representation as well as direct representation in whatever
algebra is employed for the transform.
All of the results obtained so far in hypercomplex algebras,
and known to the authors \cite[pp\,203, 209]{Hamiltonpapers:V3:7},
\cite{10.1007/s00006-006-0005-8,10.1007/s00006-010-0240-x},
were achieved by working \emph{in the algebra} in question,
that is by algebraic manipulation of quaternion,
biquaternion or Clifford multivector values.
An alternative approach would be to work in the equivalent
matrix algebra, but this seems difficult even for the lower
order cases.
Nevertheless, it merits further study
because of the possibility of finding a systematic approach
that would cover many algebras in one framework.
Following the reasoning in \S\,\ref{sec:mystery},
it is possible to define matrix roots of $-1$ that appear not
to be isomorphic to any Clifford or quaternion algebra,
and these merit further study.

Secondly, the matrix formulation presented here lends itself
to analysis of the structure of the transform, including
possible factorizations for fast algorithms, as well as
parallel or vectorized implementations for single-instruction,
multiple-data (\textsc{simd}) processors, and of course,
factorizations into multiple complex FFTs as has been done
for quaternion FFTs (see for example \cite{SangwineEll:2000b}).
In the case of matrix roots of $-1$ which do
not correspond to Clifford or quaternion algebras, analysis of
the structure of the transform may give insight into possible
applications of transforms based on such roots.

Finally, at a practical level, hypercomplex transforms implemented
{directly} in hypercomplex arithmetic are likely to be much
faster than any implementation based on matrices,
but the simplicity of the matrix exponential formulation
{discussed in this paper},
and the fact that it can be computed using standard real or complex
matrix arithmetic, {\emph{without using a hypercomplex library},}
means that the matrix exponential formulation provides a very
simple reference implementation which can be used for
verification of {the correctness of} hypercomplex code.
{This is an important point, because verification of the
correctness of hypercomplex FFT code is otherwise non-trivial.
Verification of inversion is simple enough, but establishing that
the spectral coefficients have the correct values is much less so.
}

\appendix
\section{\matlab code}
We include here two short \matlab functions
for computing the forward transform given in \eqref{eqn:forward},
and \eqref{eqn:twodforward},
\emph{apart from the scale factors}. The inverses
can be computed simply by interchanging the input and output
and negating the matrix roots of $-1$.
Neither function is coded for speed, on the contrary the coding
is intended to be simple and easily verified against the equations.
\begin{alltt}
{\color{blue}function} F = matdft(f, J)
M = size(f, 2);
F = zeros(size(f));
{\color{blue}for} m = 0:M-1
  {\color{blue}for} u = 0:M-1
    F(:, u + 1) = F(:, u + 1) {\color{blue}...}
    + expm(-J .* 2 .* pi .* m .* u./M) {\color{blue}...}
    * f(:, m + 1);
  {\color{blue}end}
{\color{blue}end}
\end{alltt}

\begin{alltt}
{\color{blue}function} F = matdft2(f, J, K)
A = size(J, 1);
M = size(f, 1) ./ A;
N = size(f, 2) ./ A;
F = zeros(size(f));
{\color{blue}for} u = 0:M-1
  {\color{blue}for} v = 0:N-1
    {\color{blue}for} m = 0:M-1
      {\color{blue}for} n = 0:N-1
        F(A*u+1:A*u+A, A*v+1:A*v+A) = {\color{blue}...}
        F(A*u+1:A*u+A, A*v+1:A*v+A) + {\color{blue}...}
        expm(-J .* 2*pi .* m .* u./M) {\color{blue}...}
        * f(A*m+1:A*m+A, A*n+1:A*n+A) {\color{blue}...}
        * expm(-K .* 2*pi .* n .* v./N);
      {\color{blue}end}
    {\color{blue}end}
  {\color{blue}end}
{\color{blue}end}
\end{alltt}

\nocite{Hamilton:1848} 
\bibliographystyle{unsrt}
\bibliography{IEEEabrv,paper}

\end{document}